\documentclass[conference,compsoc]{IEEEtran}
\pdfoutput=1
\usepackage[square,numbers]{natbib}
\usepackage{algorithm}
\usepackage{algpseudocode}

\newcommand{\R}{\mathbb{R}}
\ifCLASSINFOpdf
\else
\fi
\usepackage{macrosabound, theorem-env}
\usepackage[shortlabels]{enumitem}
\usetikzlibrary{shapes.geometric, arrows}
\tikzstyle{startstop} = [rectangle, rounded corners, 
minimum height=1cm,
minimum width=2cm,
text centered, 
draw=black, 
fill=red!30]

\tikzstyle{io} = [trapezium, 
trapezium stretches=true, 
trapezium left angle=70, 
trapezium right angle=110, 
minimum height=1cm, text centered, 
draw=black, fill=blue!30]

\tikzstyle{process} = [rectangle,
minimum height=1cm, 
text centered, 
align=center,
draw=black, 
fill=orange!30]

\tikzstyle{decision} = [rectangle, rounded corners,
minimum height=1cm, 
text centered, 
align=center,
draw=black, 
fill=green!30]
\tikzstyle{arrow} = [thick,->,>=stealth]

\makeatletter
\newcommand{\linebreakand}{%
  \end{@IEEEauthorhalign}
  \hfill\mbox{}\par
  \mbox{}\hfill\begin{@IEEEauthorhalign}
}
\makeatother

\newcommand{\MUL}[3]{\textsc{mul}[#1$\times$#2$\times$#3]}
\begin{document}
%
\title{Design of General Purpose Minimal-Auxiliary Ising Machines}

\author{
\IEEEauthorblockN{Isaac K. Martin\textsuperscript{\textdagger}}
\IEEEauthorblockA{Department of Mathematics\\
University of Texas at Austin\\
Austin, Texas\\
Email: ikmartin@utexas.edu}
\and
\IEEEauthorblockN{Andrew G. Moore\textsuperscript{\textdagger}}
\IEEEauthorblockA{Department of Mathematics\\
University of Texas at Austin\\
Austin, Texas\\
Email: agmoore@utexas.edu}
\linebreakand
\IEEEauthorblockN{John T. Daly}
\IEEEauthorblockA{Advanced Computing Systems\\
Laboratory for Physical Sciences\\
Catonsville, Maryland\\
Email: jtdaly3@lps.umd.edu}
\and
\IEEEauthorblockN{Jess J. Meyer}
\IEEEauthorblockA{Advanced Computing Systems\\
Laboratory for Physical Sciences\\
Catonsville, Maryland\\
Email: jess@lps.umd.edu}
\and
\IEEEauthorblockN{Teresa M. Ranadive}
\IEEEauthorblockA{Advanced Computing Systems\\
Laboratory for Physical Sciences\\
Catonsville, Maryland\\
Email: tranadive@lps.umd.edu}
}

%


\maketitle
\begingroup\renewcommand\thefootnote{\textdagger}
\footnotetext{These authors contributed equally.}
\endgroup
\begin{abstract}
    Ising machines are a form of quantum-inspired processing-in-memory computer which has shown great promise for overcoming the limitations of traditional computing paradigms while operating at a fraction of the energy use. The process of designing Ising machines is known as the reverse Ising problem. Unfortunately, this problem is in general computationally intractable: it is a nonconvex mixed-integer linear programming problem which cannot be naively brute-forced except in the simplest cases due to exponential scaling of runtime with number of spins. We prove new theoretical results which allow us to reduce the search space  to one with quadratic scaling. We utilize this theory to develop general purpose algorithmic solutions to the reverse Ising problem. In particular, we demonstrate Ising formulations of 3-bit and 4-bit integer multiplication which use fewer total spins than previously known methods by a factor of more than three. Our results increase the practicality of implementing such circuits on modern Ising hardware, where spins are at a premium.
\end{abstract}


%
\IEEEpeerreviewmaketitle

\section{Introduction}
The limitations of the von Neumann model of computing become clearer with each passing year. Therefore, it is important to explore both potential unconventional theoretical models of future computing and the hardware techniques which could enable their implementation. This paper will focus on the theoretical design of general-purpose Ising machines by attempting to specify quadratic Hamiltonians with arbitrary prescribed ground states and minimal auxiliary spins. This model is closely related to reversible Boltzmann machines, adiabatic quantum computing, and classical Hopfield artificial neural networks \cite{Cai_2023}; in fact, Ising Hamiltonians can be used to create algorithms for quantum computers \cite{Andriyash_2016}. 

Most work in the field of Ising algorithms has been focused on reformulating NP-complete and NP-hard optimization problems as the minimization of Ising Hamiltonians, including the travelling salesman, max-cut, and knapsack problems \cite{Lucas_2014}. Indeed, the max-cut problem has become a standard benchmark for physical Ising-type hardware. We are instead interested in creating Ising circuits which implement arbitrary logical functionality, especially integer multiplication. Previously, it has been demonstrated that arbitrary logic gates and ripple-carry addition circuits can be encoded as quadratic Ising Hamiltonians with minimal auxiliary spins \cite{Whitfield_2012} and hence that Ising-type systems are capable of universal computation \cite{Gu_2012}. However, the practical design of multiplication circuits turns out to be much more difficult than addition. 

The minimization of Ising Hamiltonians can be achieved with a wide variety of hardware, including optical coherent Ising Machines \cite{McMahon_2016}, simulation on D-Wave quantum annealers \cite{Andriyash_2016, Bian_2014}, digital FPGA implementations \cite{Patel_2020}, trapped ions \cite{Monroe_2021}, and analog oscillators with sub-harmonic injection locking \cite{Chou_2019, Wang_2019}. Each implementation technology has its own set of advantages and restrictions, but as all are still in the early stages of research, we are not concerned with working to specific architectural requirements, but rather with establishing general theory. As such, we will not attempt to minimize interaction strength dynamic range, ground state energy gap, or graph connectivity. The Ising systems discussed in this paper are zero-temperature infinite range classical spin glasses with real number interaction strengths.

The contributions of this paper are two-fold. The first is a new mathematical theory of Ising circuits which adds clarity to the problem, reveals connections to Boolean analysis and machine learning, and yields powerful theoretical results that dramatically reduce the complexity of designing Ising circuits. The second is a pair of nondeterministic algorithms which exploit our mathematical theory; combined with a number of notable optimizations, they can be applied to solve arbitrary reverse Ising problems. Section \ref{sec:theory} details the development of our theory and the proofs of our core results. As a showcase of its utility, we include a new constructive algorithmic proof of the universality of Ising systems which follows immediately from results in pseudo-Boolean optimization. Section \ref{sec:algorithms} deals with the implementation of our algorithms, the design of an optimized Mehrotra predictor-corrector method, and a number of optimizations derived from the structure and symmetry of the reverse Ising problem. In Section \ref{sec:results} we apply our theory and algorithms to the 3-bit and 4-bit integer multiplication circuits. Our solutions reduce the total number of spins required to implement these circuits by a factor of more than 3 compared to previous work.

\section{Theory}\label{sec:theory}
Here we establish a theoretical framework of Ising circuits and the reverse Ising problem. Ising spins are objects with a binary state space $\Sigma$, variously referred to as `up and down' or `1 and $-1$' or `1 and 0'. Different formulations are convenient for different situations, and when it is relevant we will explicitly state whether we are using $\Sigma = \{-1,1\}$ or $\Sigma = \{0,1\}$ convention. Note that though all qualitative results are interchangeable under a change of variables. 

\subsection{The Reverse Ising Problem}\label{sec:reverse_ising_problem}
Throughout this paper we replace the lattice $\Lambda$ in the traditional Ising model with an arbitrary finite set $X \subset \mathbb N$ and define $\Sigma^X$ to be the set of all functions $X\to \Sigma$. Such a function $\sigma$ assigns an Ising spin state to each element of $X$, and hence $\Sigma^X$ the set of all possible Ising states or configurations of $X$. For this reason we call $\Sigma^X$  the \textit{spin space} of $X$. See \cite{history} for historical conventions. When $A\subseteq X$ is a subset, for every spin state $\sigma \in \Sigma^X$ the restriction $\sigma|_A$ is an element of $\Sigma^A$ and thus represents the state of the subset $A$.

Since we are primarily interested in the spin space $\Sigma^X$, it might seem more natural to instead take $X$ to be some positive integer and $\Sigma^X$ to be the collection of vectors in $\mathbb R^{|X|}$ valued in $\Sigma = \{0,1\}$ or $\{\pm 1\}$. This convention better reflects the implementation of our algorithms but introduces the need to carefully track coordinates. We default to the coordinate-free approach as it streamlines notation and is far more convenient for mathematical formalism. The two conventions are nonetheless entirely equivalent and one can freely move between them by replacing $X$ with its cardinality $|X|$, choosing coordinates on $\mathbb R^{|X|}$, and identifying each spin state $\sigma \in \Sigma^X$ with a corresponding vertex of the hypercube.

\subsubsection{Circuits} A \textbf{circuit} is a triple $(N,M,f)$ where $N,M\subseteq X$ are disjoint subsets of $X$ called the collections of \textit{input} and \textit{output} spins respectively, and $f$ is the logic function $f:\Sigma^N\to \Sigma^M$ mapping spin states of $N$ to spin states of $M$. The collection $A = X \setminus (N\cup M)$ of spins which are neither input nor output spins is called the set of \textit{auxiliary spins}. The spinspace $\Sigma^X$ can now be canonically identified with $\Sigma^N \times \Sigma^M\times \Sigma^A$ by identifying each spin state $\sigma \in \Sigma^X$ with $(\sigma|_N, \sigma|_M, \sigma|_A)$.

For a fixed $\sigma \in \Sigma^X$ it is sometimes useful to consider the collection $\cL_\sigma = \{\sigma' \in \Sigma^X \mid \sigma'|_N = \sigma|_N\}$ of all spin states matching $\sigma$ in the input component. We call $\cL_\sigma$ the $\sigma$-\textit{input level} of the circuit $(N,M,f)$, or when the choice of circuit $(N,M,f)$ is clear, simply the \textit{input level of} $\sigma$.

\subsubsection{Ising Systems}\label{subsubsec:ising_systems} An \textbf{Ising system} is a pair $(X,H)$ where $X$ is a finite set whose elements are called \textit{spins} and $H:\Sigma^X\to \mathbb R$ is a \textit{quadratic pseudo-Boolean polynomial} (see Section \ref{subsec:pseudo_boolean}) called the \textit{Hamiltonian of $X$}. Classically, the linear coefficients of the Hamiltonian are called \textit{local biases} while the quadratic coefficients are called \textit{coupling coefficients}. 
The likelihood of observing an Ising system in a state $\sigma \in \Sigma^X$ at a temperature $T$ is given by the \textbf{configuration probability} or \textbf{Boltzmann probability}
\begin{align}
  P_\beta(\sigma) = \frac{e^{-\beta H(\sigma)}}{Z_\beta}
\end{align}
where $\beta = (k_BT)^{-1}$ is the inverse temperature, $k_B$ is the Boltzmann constant, and the normalization constant $Z_\beta$ is the partition function $Z_\beta = \sum_{\sigma \in \Sigma^X} e^{-\beta H(\sigma)}$. In the low-temperature limit, the probability that the system will be in its ground state is 1. 

There are many ways to write the Hamiltonian $H$ of an Ising system $(X,H)$. We make quick note of the two most useful conventions here.
\begin{enumerate}
    \item For distinct $i,j\in X$ denote by $h_i$ the local bias of $i$ and by $J_{ij}$ the coupling strength of $i$ and $j$. For a spin state $\sigma\in \Sigma^X$ we can then write
    \begin{align}\label{eqn:hJ_convention}
        H(\sigma) = \sum_{i\in X} h_i\sigma(i) ~+~ \sum_{i < j}J_{ij}\sigma(i)\sigma(j).
    \end{align}
    \item Interpreting $\sigma\in \Sigma^X$ as a vector with entries in $\pm 1$ and denoting by $\sigma\otimes \sigma$ the outer product, define the \textit{virtual spin} of $\sigma$ to be $\sigma$ concatenated with its $\sigma \otimes \sigma$: $v(\sigma) = (\sigma, \sigma\otimes \sigma)$. We can then write $H(\sigma)$ as the inner product
    \begin{align}\label{eqn:virtual_spin_convention}
        H(\sigma) = \langle u, v(\sigma)\rangle,
    \end{align}
    where $u$ is the \textit{coefficient vector} of $H$\footnote{As stated, these two conventions for the Hamiltonian are exactly equivalent if we allow a constant term added in Equation \ref{eqn:hJ_convention}. However, the outer product $\sigma \otimes \sigma$ is a symmetric matrix with nonzero diagonal, and thus appears to introduce extra terms in the Hamiltonian. These go away once you expand the inner product, in which case $J_{ij} = u_{ij} + u_{ji}$. In practice, we only consider the upper triangular portion of $\sigma\otimes \sigma$ to reduce dimensionality. We therefore think of $v$ as an embedding $\Sigma^X \hookrightarrow \Sigma^{|X| + {\binom{|X|}{2}}}$. This latter spin space is not physical, hence the term ``virtual spin''.}.
\end{enumerate}

\subsubsection{The Reverse Ising Problem} Given a circuit $(N,M,f)$ with indexing set $X$, we wish to design an Ising system $(X,H)$ whose behavior realizes $(N,M,f)$ with high probability. We will now make this problem precise.

We assume the input spin states can be fixed while the rest of the system evolves freely according to Ising dynamics\footnote{There are multiple ways to accomplish this depending on the specific hardware implementation. For instance, the input local biases can be made to dominate the other terms of the Hamiltonian or the inputs can be made to be ferromagnetic pairs.}. For an input state $\sigma\in \Sigma^N$ we therefore wish to maximize the probability that the output spins of the Ising system are found in the correct state $f(\sigma)$. In the low temperature limit $(\beta \gg 1)$ this is a simple optimization problem: find a Hamiltonian $H : \Sigma^N \times \Sigma^M \times \Sigma^A \longrightarrow \mathbb{R}$ such that for each input state $\sigma \in \Sigma^N$, whenever $\eta\in \cL_\sigma$ minimizes the Hamiltonian $H$ on the input level $\cL_\sigma$, $\eta|_M$ is equal to $f(\sigma)$:
\begin{align}
    \eta \in \argmin_{\omega \in \cL(\sigma)} H(\omega) \implies \eta|_M = f(\sigma).
\end{align}
The circuit data prescribe only a preferred output component for every input state. If we additionally have an auxiliary function $g:\Sigma^N\to \Sigma^A$ prescribing a preferred auxiliary component, then this becomes a linear programming problem in the coefficients of $H$: Given a circuit $(N,M,f)$, find an Ising system $(X,H)$ such that for every input state $\sigma \in \Sigma^N$, the following constraints are satisfied:
\begin{align}\label{eqn:weak-constraints}
    H(\sigma, \omega, \eta) > H(\sigma, f(\sigma), g(\sigma))
\end{align} 
for all $\eta \in \Sigma^A$ and all incorrect output states $f(\sigma) \neq \omega\in \Sigma^M$. These constraints can be written as the vector inequality
\begin{align}\label{eqn:constraint_matrix}
    Bu > 0
\end{align}
where $u$ is the coefficient vector of the Hamiltonian $H$ and $B$ is the \textit{constraint matrix} whose rows are given by the differences $v(\sigma,\omega, \eta) - v(\sigma, f(\sigma), g(\sigma))$.

We call the problem of finding such an auxiliary function $g$ along with a suitable Hamiltonian the \textbf{\textit{Reverse Ising Problem}}, and if the above constraints are satisfied, we say $(X,H)$ \textbf{solves} $(N,M,f)$. The data of $(N,M,f)$ and $(X,H)$ together is called an \textit{Ising circuit.} It is important to note that because finding $H$ is simply a linear programming problem, the determination of $g$ is the source of nearly all the difficulty. We say that a choice of $g$ is \textit{feasible} if the linear problem given in Equation \ref{eqn:weak-constraints} is feasible. Our primary goal is to find feasible choices of $g$ for a given circuit $(N,M,f)$---furthermore, we would like $A$ to be a small as possible, since spin sites are expensive. In other words, we are seeking minimal-size auxiliary maps that make feasible the realization of a given function $f$ as an Ising Hamiltonian. 

\subsection{General Observations}\label{sec:pseudo-boolean}
It is not obvious at a glance whether or not the reverse Ising problem is solvable for all circuits. In fact, it is always solvable. In this section, we cover established results from Boolean analysis which give us a constructive algorithmic proof of this result. We also discuss the relationship of Ising circuits to Hopfield networks and Support Vector Machines (SVMs), and resolve some apparent difficulties resulting from the comparison. 

\subsubsection{Universality of Quadratic Ising Systems}\label{subsec:pseudo_boolean}

We will work with $\Sigma = \{0, 1\}$. A \textit{pseudo-Boolean function} is any function of the form $f:\Sigma^n\to \mathbb R$. It is a basic result in Boolean analysis that every pseudo-Boolean functions $f$ has a unique representation as a multilinear polynomial 
\begin{align}
    \sum_{H \subseteq \{1,\dots,n\}} c_H \prod_{i \in H} x_i
\end{align}
where the $c_H$ are the Hadamard coefficients of $f$ \cite{Boros_2002}. We may therefore regard any pseudo-Boolean function as a multilinear polynomial. Each degree $n$ monomial refers to an $n$-local spin interaction, so a physically realizable Ising Hamiltonian is a quadratic pseudo-Boolean polynomial. It is easy to see that every circuit $(N, M, f)$ is solvable with a higher degree Hamiltonian polynomial $H: \Sigma^{N+M} \longrightarrow \mathbb{R}$ and zero auxiliaries by letting $H(\sigma, \eta) = d(\eta, f(\sigma))$ where $d$ is the Hamming distance (note that $H(\sigma, \eta) \geq 0$ and $H(\sigma, \eta) = 0 \iff \eta = f(\sigma)$). However, while some work has been done on higher-order spin interactions in Ising circuit design \cite{Bashar_2023}, higher-order spin interactions are generally regarded as unphysical and/or infeasible to implement. Fortunately, \textit{quadratization} techniques exist to convert higher-order polynomial minimization problems into quadratic unconstrained binary optimization (QUBO) problems. 

A \textbf{quadratization} of a pseudo-Boolean polynomial $p(\vec{x})$ is a quadratic pseudo-Boolean polynomial $q(\vec{x}, \vec{a})$ such that 
\begin{align}
\min_{\vec{a} \in \Sigma^A} q(\vec{x}, \vec{a}) = p(\vec{x}).
\end{align}
A wide variety of quadratization techniques exist \cite{Dattani_2019}, but the best general purpose algorithm in terms of minimizing the number of auxiliary values $A$ which are added is the Rosenberg reduction algorithm, which substitutes a product $x_ix_j$ with a new auxiliary variable $a$ by adding the penalty term\footnote{$C > \sum |c_H|$ is sufficient.} $P = C(xy - 2ax - 2ay + 3a)$ \cite{Rosenberg_1975, Biamonte_2008}. It is easy to check that $P \geq 0$ and $P = 0 \iff a = xy$, so substitution is guaranteed at global minima. In particular, since this process always terminates on finite polynomials, combining the previous two remarks by applying Rosenberg quadratization to the Hadamard transform of the Hamiltonian $H(\sigma, \eta) = d(\eta, f(\sigma))$ leads to the following guarantee:

\begin{prop}
    The reverse Ising problem is solvable for any circuit.
\end{prop}

\begin{example}
    It follows from an interesting quadratization result of Boros, Crama, and Rodr\'iguez-Heck \cite{boros_2020} that an Ising circuit which evaluates the parity of $n$ input bits has an elegant closed form solution. Let $p(x)$ be the parity check function on $\{0, 1\}^n$ which returns 1 if $\sum x_i$ is even and zero otherwise. Note that $\argmin_{y \in \{0,1\}} p(x,y) = p(x)$, so $p(x,y)$ is a valid higher degree Hamiltonian for the parity check circuit. Following the paper's result, this Hamiltonian can be quadratized with $\ell := \lceil \log(n+1) \rceil -1$ auxiliaries; explicitly, $H(x,y,a)$ is
    \begin{align}
        \left(\sum_{i=1}^n x_i + y + 2^{\ell +1} - \sum_{i=1}^\ell 2^i a_i - (n \mod 2)\right)^2 
    \end{align}
    In combination with the main result of this paper, this implies that parity-checking auxiliary bits can be glued to a circuit at relatively low cost. This has important applications for the implementation of LDPC (Low Density Parity Check) encoding and decoding circuits as Ising circuits \cite{Bian_2014}. 
\end{example}

\subsubsection{The Storage Capacity Paradox}\label{subsubsec:storage_capacity_paradox}

The reader who is familiar with storage capacity estimates for Hopfield networks may be somewhat suspicious that what we are doing is possible: because an Ising system as we have defined it is equivalent to storing $2^N$ patterns in a Hopfield network with $2N + M + A$ neurons, and common wisdom states that Hopfield networks have a storage capacity of $\sim 0.139k$ where $k$ is the number of neurons, then $2^N \simeq 0.139(2N + M + A)$, and hence $A \simeq 7.19(2^N - 0.139(2N + M)) \sim \mathcal{O}(2^N)$. This is not exactly a `minimal number of auxiliaries'. However, that estimate refers only to the number of `linearly independent' states which can be stored using the Hebbian learning rule \cite{Hertz_1991}. In fact, a famous result of Parisi shows that the expected number of ground states in an Ising system with i.i.d. Gaussian interaction strength is roughly $2^{0.2k}$ \cite{Posner_1985, Talagrand_2006}. In an ideal world, we could make use of all these ground states. This would mean that $2^N \simeq 2^{0.2(2N + M + A)}$, so $N \simeq 0.2(2N + M + A)$, so $A \simeq 1.6N - 0.2M$. In the case of $n\times n$ integer multiplication, $N = M = 2n$, so we would get $A \simeq 2.8n$. This back-of-the-envelope calculation shows that storage capacity bounds do not \textit{a priori} forbid the possibility of practically useful Ising circuits with quite small numbers of auxiliary spins.

\subsubsection{Single-Output Ising Circuits are SVMs}\label{subsubsec:svms_are_ising_circuits}

Intuition is greatly aided by concrete analysis of simple cases. Let $\Sigma = \{\pm 1\}$. Consider the case that $M=1$, i.e. a circuit with a single output bit. We can express any quadratic Hamiltonian as $H(\vec{x}, y) = y\tilde{H}(\vec{x}) + R(\vec{x})$ where $\tilde{H}$ is linear and $R$ is homogeneous and quadratic. If we insist that wrong states have energy at least 1 higher than correct states (as is usually done in practice for numerical convenience; see Equation \ref{eq:nonstrict}), then the constraint set for the reverse Ising problem with zero auxiliaries is
\begin{align}
    H(\vec{x}, f(x)) + 1 \leq H(\vec{x}, -f(x)) &&\forall \vec{x} \in \Sigma^N
\end{align}
Note then that 
\begin{align}
    &H(\vec{x}, f(x)) + 1 \leq H(\vec{x}, -f(x)) \\
    \iff &f(x)\tilde{H}(\vec{x}) + 1 \leq -f(x)\tilde{H}(\vec{x})\\
    \iff &f(x)(-2\tilde{H}(\vec{x})) \geq 1
\end{align}
Since $-2\tilde{H}(x)$ is linear, it is expressible as $\langle w, x\rangle - b$ for $w \in \mathbb{R}^N, b \in \mathbb{R}$. This shows that the constraint set for the reverse Ising problem is precisely the constraints for the hard-margin SVM problem, where getting the correct output is understood as a binary classification problem. Since $\Sigma^N$ embeds into $\R^N$ as the vertices of the $N$-dimensional hypercube, it follows that:

\begin{prop}\label{prop:svm_is_one_bit_Ising}
    A circuit $(N, 1, f)$ is solvable without auxiliaries if and only if $f$ is a threshold function on the $N$-hypercube.
\end{prop}
Each boolean function $f:\Sigma^d \to \Sigma$ can be thought of as a labeling of the vertices of an $d$-dimensional hypercube embedded in $\mathbb R^d$. The false set of $f$ is $f^{-1}(0)$ and likewise $f^{-1}(1)$ is the true set of $f$. We say that $f$ is a \textit{$d$-dimensional threshold function} if $f^{-1}(0)$ and $f^{-1}(0)$ are linearly separable; that is, when there exists some hyperplane $L_{\mathbf w, b} = \{\mathbf x \in \mathbb R^N ~\mid~ \langle \mathbf w, \mathbf x\rangle + b = 0\}$ such that $f^{-1}(0)$ is the set of vertices below the plane and $f^{-1}(1)$ is the set of vertices above the plane. These are well studied and have been counted up to $d=9$, see \cite{threshold_msc_thesis} for an overview and \cite{recent_threshold} for more recent results. We discuss the analog of soft-margin SVMs and various applications of this result in Section \ref{subsec:artificial_variables}.

\subsection{Augmented Constraints}\label{subsec:augmented_constraints}
Solutions to the reverse Ising problem na\"ively happen in two steps:
\begin{enumerate}
    \item The auxiliary problem: find an appropriate size for $A$ and an auxiliary function $g:\Sigma^N\to \Sigma^A$. This is a nonlinear nonconvex mixed-integer constrained optimization problem.
    \item The linear problem: solve the linear programming problem (Equation \ref{eqn:weak-constraints}).
\end{enumerate}

Early attempts at general algorithmic solutions to the reverse Ising problem attempt to make iterative improvements to an initial choice of auxiliary function $g:\Sigma^N\to \Sigma^A$ by measuring the feasibility of $g$ using some heuristic. A quick analysis reveals this to be unsuitable for all but the simplest circuits. Finding a feasible auxiliary function $g$ involves searching through the space of all possible auxiliary functions (a set of size $2^{A\cdot 2^N}$), and the lack of convexity in the auxiliary problem means that feasibility heuristics are of limited use. Worse, feasibility heuristics generally require a pass through the linear problem (see Section \ref{sec:algorithms} for details). A quick check of Equation \ref{eqn:weak-constraints} reveals the number of constraints scales exponentially in $A$---there are precisely $2^N\cdot (2^{M+A} - 2^A)$ constraints, each of length\footnote{The virtual spin has $N+M+A + \binom{N+M+A}{2}$ components, but the columns corresponding to combinations of spins in $N$ are always zero in the constraint matrix, and can therefore be removed.} $M+A + {N+M+A \choose 2} - {N \choose 2} \sim \mathcal{O}(N+M+A)^2$. Therefore the difficulty of the linear problem for assessing a candidate $g$ with respected to a fixed circuit $(N,M,f)$ grows like $\mathcal{O}(2^AA^2)$ as $A$ increases. The exponential scaling makes finding $g$ for all but the smallest circuits practically impossible. A more sophisticated approach is needed.

There are two obvious avenues for improvement: (1) cut down on the size of the auxiliary search space and (2) reduce the cost of the linear problem. Theorem \ref{thm:augmented_constraints} provides sizable improvements on both of these fronts by drastically reducing the search space of possible auxiliary functions and eliminating the exponential scaling in $A$, thus reducing the complexity of the linear problem to $\mathcal{O}(A^2)$ with respect to a fixed circuit. The key insight is to allow the auxiliary function $g$ to depend on both $\Sigma^N$ and $\Sigma^M$ rather than only $\Sigma^N$.
\begin{thm}\label{thm:augmented_constraints}
  Let $(N, M, f)$ be a circuit. There exists an Ising system which solves this circuit if and only if there is a function $g:\Sigma^N\times\Sigma^M \to \Sigma^A$ such that
  \begin{enumerate}[(a)]
    \item The circuit $(N\cup M, A, g)$ is solvable by an Ising system with Hamiltonian $R$ with the following additional property:
      \begin{equation}\label{eqn:weak-neutralizability}
        R(\sigma, \omega, g(\sigma,\omega)) \geq R(\sigma, f(\sigma), g(\sigma, f(\sigma)))
      \end{equation}
      for all input states $\sigma$ and output states $\omega$. Equation \ref{eqn:weak-neutralizability} is the \textbf{weak neutralizability condition}. If we instead have equality then it is the \textbf{strong neutralizability condition}. If such an Ising system (X,R) exists, we correspondingly say that $g$ is solvable weakly neutralizable or solvable strongly neutralizable.
    \item There is an Ising system $(X, H)$ which satisfies \textbf{$g$-augmented constraints}:
    \begin{align}\label{eqn:aug-constraints}
        H(\sigma, \omega, g(\sigma, \omega)) > H(\sigma, f(\sigma), g(\sigma, f(\sigma))).
    \end{align}
  \end{enumerate} 
We call $(X,H)$ the \textbf{base system} and the circuit $(N,M,f)$ the \textbf{base circuit}. We call the system $(X, R)$ the \textbf{auxiliary system} and the circuit $(N\cup M, A, g)$ the \textbf{auxiliary circuit}. 
\end{thm}
The proof can be found in Section \ref{subsec:moremath}.
\begin{rmk}
    Weak neutralizability depends on the logic of the circuit $(N,M,f)$; an auxiliary functions $g$ may be weakly neutralizable for some choices of $f$ but not for others. Strong neutralizability does not depend on $f$.
\end{rmk}
\begin{rmk}\label{rmk:gluing_approach}
    At first glance it seems we have made the situation worse -- Theorem \ref{thm:augmented_constraints} splits the task of finding a single Ising system which solves $(N,M,f)$ into the task of finding two Ising systems satisfying distinct constraint sets. The advantage becomes clear given the following two observations:
  \begin{enumerate}
    \item \textbf{Improvements to $g$ can be made iteratively.} It is easy to check that if both $(N\cup M, A_1, g_1)$ and $(N\cup M, A_2, g_2)$ are solvable weakly neutralizable auxiliary circuits, then they can be "glued" together to form a new circuit $(N\cup M, A_1\sqcup A_2, g_1\times g_2)$ where
      \begin{align*}
        (g_1\times g_2)(\sigma, \omega) := (g_1(\sigma, \omega), g_2(\sigma, \omega)) \in \Sigma^{A_1 \cup A_2},
      \end{align*}
      itself a solvable weakly neutralizable auxiliary circuit.
    \item \textbf{Augmented constraints do not scale exponentially in $A$.} The traditional constraints (Equation \ref{eqn:weak-constraints}) grow exponentially in $A$, but the constraint matrices of the $g$-augmented constraints \ref{eqn:aug-constraints}) have no row-wise dependence on $A$ and only grow quadratically columnwise.
  \end{enumerate}
  Observation (1) means the function $g$ can be built from libraries of known solvable weakly neutralizable functions, which shrinks the space of possible auxiliary functions so drastically that heuristic-driven brute force searches become computationally viable. Observation (2) means that the complexity of the linear solve grows only quadratically in $A$, and hence feasibility criterion are far cheaper to compute.

  We call this the \textit{gluing approach} to the reverse Ising problem.
\end{rmk}
The following two example auxiliary functions are practical in application.
\begin{example}\label{ex:trivial_weak_neutralizability}
  Suppose $g:\Sigma^N \times \Sigma^M \to \Sigma^A$ is constant in the $\Sigma^M$ component. Then any Ising system $(X, R)$ which solves the circuit $(N\cup M, A, g)$ is trivially strongly neutralizable.
\end{example}
\begin{example}\label{ex:AND_gate_strong_neutralizable}
  Let $(\{a,b\}, \{c\}, AND)$ be the 1-bit AND circuit. There exists an Ising system $(\{a,b,c\}, R)$ which solves the circuit and is strongly neutralizable. Since the substitution $a$ for $xy$ is a logical AND gate, Rosenberg reduction can therefore be viewed as the construction of solvable strongly neutralizable auxiliary functions through the successive gluing of AND gates.
\end{example}

\section{Algorithms and Optimizations}\label{sec:algorithms}
Having established our approach to the reverse Ising problem, we now turn to the task of utilizing this theory to producing concrete solutions to specific circuits. Armed with Theorem \ref{thm:augmented_constraints}, we will discuss implementations and improvements the basic approach described at the start of Section \ref{subsec:augmented_constraints}:
\begin{enumerate}
    \item Make an initial guess $g:\Sigma^N\times \Sigma^M\to \Sigma^A$ of neutralizable auxiliary function.
    \item Measure the feasibility of $g$ using a linear programming solver on the augmented constraints.
    \item If $g$ is feasible then we are done. Otherwise update the choice of $g$.
\end{enumerate}
In Section \ref{subsec:artificial_variables} we discuss a modification of the linear problem which yields a feasibility heuristic useful for guiding the reassignment of $g$. Section \ref{subsec:main_search_algos} covers the main search algorithms used to construct a feasible $g$. Section \ref{subsec:linear_solver} describes the design of a bespoke linear solver optimized for our specific problem. Sections \ref{subsec:symmetry_speedup} and \ref{subsec:constraint_filtering} detail two further optimizations which offer significant improvements to the search algorithms.

\subsection{Linear Problem with Artificial Variables}\label{subsec:artificial_variables}

Recall the statement of our problem: we wish to find $g : \Sigma^{N+M} \longrightarrow \Sigma^A$ and a quadratic pseudo-Boolean polynomial $H : \Sigma^{N+M+A} \longrightarrow \mathbb{R}$ such that
\begin{align}
\label{eq:nonstrict}
    H(\sigma, f(\sigma), g(\sigma, f(\sigma))) + 1 \leq H(\sigma, \omega, \eta)
\end{align}
for $\omega \neq f(\sigma)$. By the main theorem, this can be reduced to the set of constraints
\begin{align}
    H(\sigma, f(\sigma), g(\sigma, f(\sigma))) + 1 \leq H(\sigma, \omega, g(\sigma, \omega))
\end{align}
as long as $g$ is weakly neutralizable. Writing the Hamiltonian $H$ as the inner product $H(\sigma) = \langle u, v(\sigma)\rangle$ as Equation \ref{eqn:virtual_spin_convention}, the constraints become 
\begin{align}
    \big\langle u, ~v(\sigma, f(\sigma), g(\sigma, f(\sigma))) - v(\sigma, \omega, g(\sigma, \omega))\big\rangle \leq -1 \\
    \forall \sigma \in \Sigma^N, \omega \in \Sigma^M, \omega \neq f(\sigma)
\end{align}
Therefore, letting $B(f,g)$ be the constraint matrix whose rows are $v(\sigma, \omega, g(\sigma, \omega)) - v(\sigma, f(\sigma), g(\sigma, f(\sigma)))$, we wish to find a vector $u$ such that $B(f,g)u \geq 1$. Since most choices of $g$ will yield infeasible problems, we need a choice of \textit{feasibility heuristic} to evaluate how close $g$ is to generating a feasible problem. This can be done by adding new vector of variables $\rho$ called the \textbf{artificial variables} to obtain the linear programming problem
\begin{align}\label{eqn:artificial_variables}
    \varrho(f,g) := \min_{\rho, u} \|\rho\|_1 && \text{ s.t. } B(f,g)u + \rho \geq 1, \rho \geq 0
\end{align}
The optimal value of the objective function is zero if and only if the choice of $g$ is feasible. Otherwise, the optimal value of the objective function $\varrho(f,g)$ roughly measures how infeasible $g$ is.

\begin{rmk}
    In Section \ref{subsubsec:svms_are_ising_circuits}, we discussed the fact that the linear problem for $M=1$ is equivalent to fitting a linear hard-margin SVM. Adding the artificial variables to the $M=1$ Ising circuit is in fact exactly equivalent to the linear soft-margin SVM.
\end{rmk}

\subsection{Main Search Algorithms}\label{subsec:main_search_algos}

We will use the notation $\times$ to denote the concatenation of tuples. Armed with the artificial variables heuristic, we can attempt to build a weakly neutralizable auxiliary map $g$ that makes a given circuit $(N, M, f)$ feasible. The simplest choice is a greedy algorithm based on threshold functions.
\begin{algorithm}
\caption{Greedy Search}
\label{alg:greedy}
\begin{algorithmic}
    \Require $\{\mathcal{T}_n\}_{n \in \mathbb{N}}$ sets of strongly neutralizable threshold functions on $n$ variables.
    \Require $(N, M, f)$ Ising circuit.
    \State$g \gets \emptyset$
    \While{$\varrho(f,g) > 0$}
        \State$g \gets g \times\argmin_{a \in \mathcal{T}_{N+M+|g|}} \varrho(f,g\times a)$
    \EndWhile
\end{algorithmic}
\end{algorithm}

\noindent
It follows from Section \ref{subsec:pseudo_boolean} that if 
\begin{align}
    \mathcal{T}_n \supseteq \mathscr{A}_n := \{x \longmapsto x_i \wedge x_j | 1 \leq i < j \leq n\}
\end{align}
then Algorithm $\ref{alg:greedy}$ always terminates, since $\varrho(f,g) \geq \varrho(f, g\times a)$, for we can always set the coefficients on the new threshold function $a$ to zero and recover the left hand side. It should be noted that requiring $\mathcal{T}_n$ grow with $n$ is required for the algorithm to always terminate---fixing $\mathcal{T}_n = \mathscr{A}_{N+M}$ results in the algorithm never terminating when attempting to solve the parity check circuit with 3 input bits. 
\begin{rmk}\label{rmk:strong_neutralizability_saves_overhead}
    The $\mathcal T_n$ sets must be computed separately prior to the execution of Algorithm \ref{alg:greedy}, see \ref{subsubsec:build_libraries_of_snt} for details. We use strongly neutralizable threshold functions for convenience; since the strong neutralizability property has no dependence on the underlying base circuit $(N,M,f)$, and the $\mathcal T_n$ don't change from problem to problem. Using weakly neutralizable threshold functions should in principle yield a solution with fewer auxiliaries, since it gives many more options and thus increases the flexibility of the search. However, this adds significant overhead as the $\mathcal T_n$ sets must be computed for each new problem.

    Note also that any auxiliary function $g$ can be written componentwise as $(g_1,...,g_A)$, where each $g_i:\Sigma^N\times \Sigma^M\to \Sigma$ represents a single auxiliary bit. Since $g$ is a solvable weakly/strongly neutralizable auxiliary function if its component functions are weakly/strongly neutralizable threshold functions, building $g$ from its components always produces viable auxiliary circuits.
\end{rmk}

In practice, the greedy algorithm is clearly non-optimal. This is because $\varrho$ is not additive with respect to its components: threshold functions have unpredictable synergies with each other and thus cannot be regarded as independent. That is, for threshold functions $a$ and $b$, $\varrho(f, \emptyset) - \varrho(f, a \times b)$ differs unpredictably from $2\varrho(f, \emptyset) - \varrho(f, a) - \varrho(f, b)$, though the difference is usually not too large. Therefore, the greedy algorithm may result in earlier choices becoming non-optimal after later choices are made. This suggests a `coordinate descent' type algorithm, Algorithm \ref{alg:descent}, which takes further advantage of the artificial variables heuristic to continually revise choices towards more optimal solutions. Additionally, it makes sense to expand the search space to weakly neutralizable functions, since gluing works in the exact same way.

\begin{algorithm}
\caption{Descent Algorithm}
\label{alg:descent}
\begin{algorithmic}
    \Require $(N,M,f)$ Ising circuit.
    \Require $\mathcal{T}$ set of weakly neutralizable threshold functions with respect to circuit $(N, M, f)$. 

    \State $R \gets (g, j, a) \longmapsto \left(\hat{g}_i = \begin{cases} 
        g_i &\text{ if } i \neq j\\
        a &\text{ if } i = j
    \end{cases}
    \right)_{1 \leq i \leq |g|}$
    \State $g \gets \emptyset$
    \State $S \gets \{1, \dots, |g|\}$
    \While{$\varrho(f, g) > 0$}
        \If{$S = \emptyset$}
            \State $g \gets g \times \argmin_{a \in \mathcal{T}} \varrho(f,g\times a)$
            \State $S \gets \{1, \dots, |g|\}$
        \EndIf
        \State $j \gets \argmin_{i \in S} \varrho(f, g \setminus \{g_i\}) - \varrho(f, g)$
        
        \State $\ell \gets N+M+j-1$
        \State $\alpha \gets \argmin_{a \in \mathcal{T} \cup \mathscr{A}_\ell} \varrho(f, R(g,j,a))$ 
        \If{$\varrho(f,R(g,j,\alpha)) < \varrho(f,g)$}
            \State $g_j \gets \alpha$
            \State $S \gets \{1, \dots, |g|\}$
        \Else
            \State $S \gets S \setminus \{j\}$
        \EndIf
    \EndWhile
\end{algorithmic}
\end{algorithm}

This is a much better algorithm in practice, though it will not always produce solutions with minimum $|g|$ due to getting stuck in local minima and therefore being forced to increase $|g|$. Due to the inclusion of the $\mathscr{A}_\ell$ sets, this algorithm will always terminate in finite time for the same reason as Algorithm \ref{alg:greedy}. Plenty of tweaks can be made to the algorithm, most resulting in greater complexity:
\begin{itemize}
    \item The heuristic values which are optimized over to set $j$ are designed to select the auxiliary functions which are contributing the least---a more precise but far more computationally expensive option would be setting $j$ by optimizing over the Shapley numbers of the $g_i$. Fast approximations of Shapley numbers do exist, but we leave experimentation with their implementation for future research.
    \item The symmetries discussed in Section \ref{subsec:symmetry_speedup} can be leveraged, in combination with a cache, to reduce the number of times that the expensive function $\varrho$ is called (see Remark \ref{rmk:symmetry_speedup}).
    \item The filtering method discussed in Section \ref{subsec:constraint_filtering} can be implemented: We start with $\varrho = \varrho_i$ and increment $i$ every time the condition $\varrho(f,g) = 0$ is satisfied, breaking out of the while loop only when $i = M$. This also significantly speeds up execution time.

\end{itemize}

\subsubsection{Building Libraries of Threshold Functions}\label{subsubsec:build_libraries_of_snt} 

Remark \ref{rmk:strong_neutralizability_saves_overhead} suggests that we precompute libraries of strongly neutralizable threshold functions for use in Algorithms \ref{alg:greedy} and \ref{alg:descent}. This is done by first finding all threshold functions $f:\Sigma^d\to \Sigma$ of a fixed dimension $d$ and then exhaustively checking the strongly neutralizability condition.

There are two ways we find threshold functions of a fixed dimension $d$. This is certainly suboptimal but has thus far proven sufficient for our approach.
\begin{enumerate}
    \item The set $\Sigma^d$ can be identified with the set of integers $[0..2^d-1] = \{n \in \mathbb Z \mid 0\leq n< 2^d\}$. In this way, each Boolean function $f:\Sigma^d \to \Sigma$ can be thought of as a function $f:[0..2^d-1] \to \Sigma$, and thus corresponds to a binary string of length $2^d$ given by $[f(0), f(1), ..., f(2^d - 1)]$. There are $2^{2^d}$ such binary strings, so it quickly becomes infeasible to check them all for linear-separability. A laptop can nonetheless handle the case of $d=5$, $2^{2^5} = 4,294,967,296$ without much difficulty, especially once symmetries are utilized (see Section \ref{subsec:symmetry_speedup}).
    \item Threshold functions of dimension $d$ are linear separations of the $d$-dimensional hypercube, hence any plane $L_{\mathbf w, b}: \langle \mathbf w, \mathbf x\rangle + b = 0$ in $\mathbb R^d$ defines a threshold function. By randomly sampling $\mathbf w$ and $b$, caching functions, and comparing to proven threshold function counts (see \cite{oeis_threshfunc} for instance) one can quickly identify the majority of threshold functions up to dimension 7. 
\end{enumerate}
Both of these approaches are sped up by exploiting symmetry, see Section \ref{subsec:symmetry_speedup}.

\subsection{Linear Programming Solver}\label{subsec:linear_solver}

We found that no freely available linear programming solver was sufficiently fast or memory efficient to tackle the high volume of large-size linear problems needed to calculate the feasibility heuristic in our search algorithms. We wrote a multithreaded C implementation of the Mehrotra predictor-corrector based on \cite{Nocedal_2006} specifically optimized for our problem. It is 2-3 times faster than GLOP \cite{ortools} for highly overdetermined sparse sign matrices like the Ising constraint sets, and significantly more memory efficient. Details of the solver design can be found in Section \ref{sec:solver}. It is inspired by the work in \cite{teresa}, but trades away memory optimization for greater flexibility in choosing the coefficient matrix, among several other modifications designed to suit our problem in particular. 

\subsection{Symmetries of Ising Circuits}\label{subsec:symmetry_speedup}
We set $\Sigma = \{\pm 1\}$ unless otherwise noted. Fix a circuit $(N,M,f)$. It is clear that an auxiliary function $g:\Sigma^{N}\times \Sigma^M\to \Sigma^A$ is solvable and weakly neutralizable if and only if $-g$ is solvable and weakly neutralizable; if the Hamiltonian $H$ solves the circuit with auxiliary function $g$ then the Hamiltonian $H'$ obtained from $H$ by flipping the signs of every coefficient of an auxiliary local bias, auxiliary/input interaction or auxiliary/output interaction solves the circuit with auxiliary function $-g$. Hence if $g$ is determined to be infeasible then $-g$ must also be infeasible without performing any additional computation. It is natural to ask: are there other such transformations which preserve Ising solvability?

Let us make this situation more precise. The transformations $g\mapsto -g$ and $H\mapsto H'$ can be thought of a single transformation which satisfies the following compatibility condition: if $(X,H)$ solves the auxiliary circuit $(N\cup M, A, g)$ then $(X, H')$ solves $(N\cup M, A, -g)$. Generalizing, we wish to understand objects $\alpha$ which act on both the logic of circuits and Ising Hamiltonians and which preserve solvability. Such an object $\alpha$ is called an \textit{Ising circuit symmetry}.

There are two types of Ising circuit symmetries we consider: \textit{spin actions} and \textit{input permutations}. Relevant proofs of the details in this section are given in the Appendix.

\subsubsection{Spin Actions} Viewing $\Sigma^X$ as the set of functions $X\to \Sigma$, for each $\alpha\in \Sigma^X$ we obtain an action on $\Sigma^X$ via pointwise multiplication by $\alpha$:
\begin{align}
    \alpha\sigma (x) = \alpha(x)\cdot \sigma(x).
\end{align}
This is called a \textit{spin action}, and if spin states are instead viewed as tuples of $\pm 1$ then this is nothing more than the component-wise multiplication of $\{\pm 1\}^n$. Decomposing $\alpha$ into input and output components gives an action on $f$ by
\begin{align}
    (\alpha f)(\sigma) := \alpha|_M \cdot f(\alpha|_N\sigma).
\end{align}
We get a corresponding action on Ising systems $(X,H)$ by multiplying $v(\alpha)$ componentwise with $u$, the coefficient vector of $H$:
\begin{align}
    \alpha u := v(\alpha)\cdot p.
\end{align}
Thus defined, $\alpha$ is an Ising circuit symmetry, see Lemma \ref{prop:spin_actions_are_symmetries} for proof. The group of all spin actions on $\Sigma^X$ is easily seen to be isomorphic to $(\mathbb Z/2\mathbb Z)^n$, where $|X| = n$.

\subsubsection{Coordinate Permutations}\label{subsubsec:coordinate_permutations} Given a permutation $\alpha$ of $N$ and a permutation $\beta$ of $M$, we can define an action on $f:\Sigma^N\to \Sigma^M$ by
\begin{align}
    (\beta f \alpha)(\sigma) = f(\sigma \circ \alpha) \circ \beta.
\end{align}
The corresponding action on the coefficient vector $p$ of $H$ is easier to write using the $h$ and $J$ convention (see Equation \ref{eqn:hJ_convention}). We define the action of $\alpha$ on $h$ and $J$ to be the identity action, and we define $\beta$ to act by index permutation
\begin{align}
    \beta h_i := h_{\beta(i)}, \quad \beta J_{i,j} := J_{\beta(i),\beta(j)}
\end{align}
setting $\beta(i) = i$ whenever $i \in N$. This makes coordinate permutations Ising circuit symmetries; see Lemma \ref{prop:coordinate_permutations_are_symmetries}. The group of all coordinate permutations is the product of permutation group of orders $N$ and $M$: $S_N\times S_M$.

\subsubsection{Symmetry Speedups}
Denote by $\mathcal G(N,M)$ the group of \textit{Ising symmetries} obtained by compositions of spin actions and coordinate permutations for circuits of the form $(N,M,f)$ and by $\mathcal S(N,M)$ the group of spin actions. The following proposition and remark illustrate how symmetries of Ising circuits can speed up our algorithms.
\begin{prop}\label{prop:symmetries}
    If $\alpha \in \mathcal G(N,M)$  then $g:\Sigma^N\to \Sigma$ is a weakly/strongly neutralizable threshold function if and only if $\alpha g$ is as well. Additionally, if $\alpha\in \mathcal S(N\cup M,A)$ is a spin action on the auxiliary circuit $(N\cup M, A, g)$, then $\varrho(f,g) = \varrho(f,\alpha g)$.
\end{prop}
\begin{rmk}\label{rmk:symmetry_speedup}
    This proposition speeds up the search for auxiliary functions in two ways.
    \begin{enumerate}
        \item \textbf{Building libraries of threshold functions.} When computing sets $\mathcal T$ of threshold functions in Algorithms \ref{alg:greedy} and \ref{alg:descent}, one need only test one candidate threshold function $g$ for solvability and neutralizability to determine the solvability and neutralizability of the entire orbit. This eliminates
        \begin{align}
            |\mathcal G(N,M)| = 2^{N\cup M} \cdot N! \cdot M!
        \end{align}
        linear solves in the best case scenario, but far fewer in practice since the action of $\mathcal G(N,M)$ is not free.
        \item \textbf{Computing the feasibility heuristic.} Because spin actions preserve $\varrho$, two auxiliary functions $g$ and $g'$ related by a spin action are \textit{indistinguishable by the feasibility heuristic}. One can therefore take $\mathcal T$ to contain only one weakly/strong neutralizable threshold function in each $\mathcal S(N,M)$ orbit without losing any performance. This eliminates up to $|\mathcal S(N,M)| = 2^{N+M} - 1$ computations of $\varrho$.
    \end{enumerate}
\end{rmk}
\begin{proof}[Proof of Proposition \ref{prop:symmetries}]
    Combining Proposition \ref{prop:svm_is_one_bit_Ising} with Lemmas \ref{prop:spin_actions_are_symmetries} and \ref{prop:coordinate_permutations_are_symmetries} shows that spin actions and coordinate permutations preserve threshold functions. Slight modifications to the proofs of these two lemmas show that weak neutralizability is also preserved.

    If $\alpha$ is a spin action then $B(f,\alpha g)$ is obtained from $B(f, g)$ by taking the component-wise product of every row with $v(\alpha)$. It then follows that
    \begin{align}
        B(f,\alpha g) (\alpha u) = B(f,\alpha g) (v(\alpha) \cdot u) = B(f, g) u
    \end{align}
    since $v(\alpha) \cdot v(\alpha)$ is the vector consisting only of $1$'s. Thus the $\rho$ and $u$ which minimize Equation \ref{eqn:artificial_variables} for $g$ and $\alpha g$ are equal.
\end{proof}

\subsubsection{Toward a Classification of Strongly Neutralizable Threshold Functions} We switch now to $\Sigma = \{0,1\}$ convention. Spin actions and coordinate permutations have long been used in the classification of threshold functions under the names \textbf{u-complementation} and \textbf{permutation of variables} respectively \cite{threshold_book}. Another linear-separability preserving operation known as \textit{self-dualization} is common in threshold logic. Given a threshold function $f:\Sigma^n\to \Sigma$, the self-dualization of $f$ is a threshold function $f^{sd}$ of dimension $n+1$. If $s_0$ is used to denote the new variable, then it is defined
\begin{align}
    f^{sd}(s_0,...,s_n) = s_0f(s_1,...,s_n) + \overline{s_0}\overline{f}(\overline{s_1},...,\overline{s_n})
\end{align}
where $\overline{s_i}$ is used to denote the negation of the coordinate $s_i$. We say that a threshold function $f$ is self-dual if it is the self-dualization of a lower dimensional threshold function. Because self-dualization additionally preserves strong neutralizability but does \textit{not} preserve the artificial variable feasibility heuristic $\varrho$, it gives us access to non-redundant strongly neutralizable threshold functions in the context of Algorithms \ref{alg:greedy} and \ref{alg:descent}.

A more trivial way to produce higher dimensional strongly neutralizable threshold functions from lower dimensions is known as \textit{extrusion}. Given an $n$-dimensional threshold function $f:\Sigma^N\to \Sigma$, we may \textit{extrude} $f$ along a new dimension by simply ignoring the additional variable $s_0$:
\begin{align}
    f^e(s_0,...,s_n) := f(s_1,...,s_n).
\end{align}
This operation does preserve $\varrho$, and hence the strongly neutralizable threshold functions it produces are redundant.

If $f$ is self dual, then it is easy to check that $\overline{f}(\overline{s_1},...,\overline{s_n}) = f(s_1,...,s_n)$. This implies the second self-dualization of a Boolean function is an extrusion, as there is no dependence on the second added dimension. The only non-redundant strongly neutralizable threshold functions, therefore, are those which are either \textit{(i)} not self dual or \textit{(ii)} are the self-dualization of a non self-dual function. Up to Ising symmetry $\mathcal G(N,M)$, we have found only two non-redundant strongly neutralizable threshold functions: the AND gate in dimension $2$ and its self-dualization in dimension $3$. All others appear to be in the $\mathcal G(N,M)$-orbit of an extruded lower dimensional threshold function -- though we have only checked up through dimension 7. We conjecture the following classification of strongly-neutralizable threshold functions:
\begin{conjecture}
    The only strongly-neutralizable threshold functions of dimension 2 or greater which are not extrusions of lower dimensional functions are $\text{AND}$ and its self-dualization $\text{AND}^{sd}$ up to $\mathcal G(N,M)$ action.
\end{conjecture}

\subsection{Constraint Filtering}\label{subsec:constraint_filtering}

The linear problem of checking the feasibility of an augmented constraint set for circuit $(N,M,f)$ with $A$ auxiliaries has difficulty $\mathcal{O}(2^{N+M} (N+M+A)^2)$. If we need to search through a large amount of auxiliary functions, this may still be quite expensive. In practice, however, it turns out that the artificial variables are not evenly distributed across the rows of the constraint matrix, and therefore we can cut down on the factor of $2^M$ quite substantially.

For a fixed circuit $(N,M,f)$ and an infeasible auxiliary function $g:\Sigma^{N+M}\to \Sigma^A$, the constraint matrix $B$ serves only to verify the infeasibility of $g$ and compute the criterion $\varrho$. This matrix is quite tall: $B$ has $2^N\cdot (2^M - 1)$ rows but only $M+A + \binom{N+M+A}{2} - \binom{N}{2}$ columns. This results in an exceptionally over-determined system in which we expect many redundant constraints. Now consider a matrix $B'$ consisting of only a subset of the rows in $B$ and the reduced problem $B'u\geq 1$ with feasibility criterion $\varrho'$. Evidently $\varrho'$ is a lower bound on $\varrho$ which is  cheaper to compute than $\varrho$. Therefore we may reasonably want to know: \textit{(i)} How well does $\varrho'$ approximate $\varrho$, as a function of the number of constraints removed from $B$---and relatedly, \textit{(ii)} If $\varrho' = 0$, what is the probability that $\varrho = 0$? If it is in fact a good approximation, we can use it as an approximate criterion for the first phase of the algorithm and thus save a lot of time. This is indeed the case: Figure \ref{fig:rhoi} demonstrates that smaller constraint sets serve as good approximations of \ref{fig:rhoi} and Figure \ref{fig:minnumconstraints} illustrates that they detect infeasibility with high probability.

As the size of $A$ is increased (i.e. as threshold functions are added in Algorithms \ref{alg:greedy} and \ref{alg:descent} and $g$ becomes more feasible) we expect that the requisite size of the matrices in this filtration will also grow. It is therefore useful to consider a series of smaller linear problems of various sizes obtained by incrementally adding constraints from $B$; that is, to consider a filtration $\{B_i\}_{i\in [0..k]}$ of the constraint matrix $B$ rather than a single reduced matrix $B'$.

Each row of $B$ is defined by a choice of input $\sigma \in \Sigma^N$ and a choice of incorrect output $\omega \in \Sigma^M, \omega \neq f(\sigma)$. In practice, rows whose incorrect output $\omega$ is closer in Hamming distance to the correct output $f(\sigma)$ contribute more to the heuristic $\varrho$ and are collectively more difficult to satisfy. Therefore, a natural choice of filtration is $\{B_i\}_{1 \leq i \leq M}$, where $B_i$ is the matrix consisting of all rows defined by choices of incorrect output $\omega$ such that $d(\omega, f(\sigma)) \leq i$ and $d$ denotes Hamming distance. We call $B_i$ the `$i$th local constraints' because it requires that $(\sigma, f(\sigma))$ be the absolute minimum of the input level $\mathcal L_\sigma$ on the Hamming ball of radius $i$ around $\sigma$, for every $\sigma$. Note that $B_1$ has $M$ rows per input level, and $B_i$ adds $\binom{M}{i}$ rows to each input level relative to $B_{i-1}$.

It turns out that the artificial variables are mostly concentrated in the local constraints with $i$ small (see Figure \ref{fig:rhoi}), which also have far fewer rows. Therefore, testing with the first or second local constraints provides a reasonably good lower bound to the total sum of artificial variables for the whole problem at a fraction of the computational cost. If we filter possible candidate auxiliary function candidates with the $X_1$ constraint set, each linear program solve takes has difficulty $\mathcal{O}(2^NM(N+M+A)^2)$, an improvement over the whole problem by a factor of $2^M/M$. In practice, we first construct a solution which satisfies $B_1$ or $B_2$ which we then use as a starting point to search for a solution to $B$, thus saving a significant amount of effort.

\begin{figure}[h]
    \includegraphics[width=0.99\linewidth]{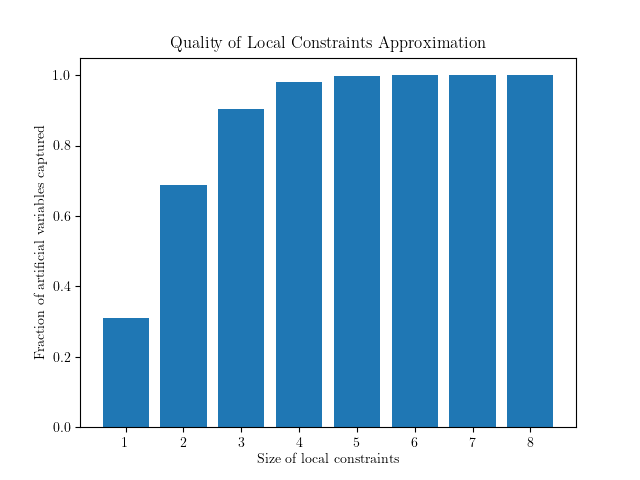}
    \caption{Proportion of artificial variables captured by the $i$th local constraints, $\varrho_i/\varrho$, plotted against $i$ for the problem of $4 \times 4$ integer multiplication with 12 auxiliaries, sampled randomly from a set of weakly neutralizable threshold functions, averaged over 50 runs. Observe that $i=3$ captures almost all of the artificial variables.}
    \label{fig:rhoi}
\end{figure}
\begin{figure}[h]
    \includegraphics[width=0.99\linewidth]{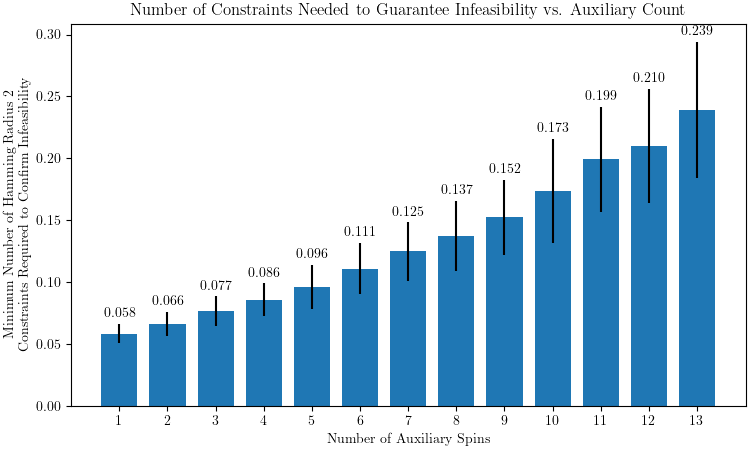}
    \caption{The minimum number of hamming radius 2 constraints required to confirm the infeasibility of randomly chosen auxiliary functions of specified sizes, averaged across 500 runs. The minimum number of constraints needed to detect infeasibility was obtained by performing a binary search on a random maximal filtration of $B_2$, the 2nd local constraints. We emphasize that the $y$-axis measures the number of requisite constraints as a fraction of all 2nd local constraints; the collection of all 2nd local constraints themselves only account for $\frac{9216}{65280} \approx 0.1412$ of the total number of constraints.}
    \label{fig:minnumconstraints}
\end{figure}

\section{Results}\label{sec:results}

Our chosen benchmark problem is the implementation of integer multiplication circuits. We denote the problem of finding sets of $a$ threshold functions which solve the problem of $n \times m$ integer multiplication as \MUL{$n$}{$m$}{$a$}. Our methods were successful in producing many solutions to \MUL{3}{3}{3}, \MUL{3}{4}{5}, and \MUL{4}{4}{12}. In each case, our solution represents the current optimum in total spin count by a large margin. Previously, Andriyash~\cite{Andriyash_2016} constructed solutions to \MUL{3}{3}{42} and \MUL{4}{4}{88} on a D-Wave system. Our results represent a significant reduction in total circuit size, from $54$ to $15$ total spins for $3 \times 3$ multiplication and from $104$ to $28$ total spins for $4 \times 4$. Therefore, our results reduce the total circuit layout area by a factor of around $3.5$ compared to previous designs. By reducing the total number of spins substantially, we have decreased the cost of implementing such circuit in hardware.

\subsection{Runtimes \& Example Data}

Experiments were run on dual socket compute nodes with 64 cores per AMD EPYC 7713 socket and two threads per core for a total of 256 processors per node. Average runtimes for successful solutions to our benchmark problems were as follows:
\begin{center}
\begin{tabular}{|c|c|c|c|}
\hline
Problem & Average Runtime (s)\\
\hline
\MUL{3}{3}{3} &  902\\
\MUL{3}{4}{5} & 23,520\\
\MUL{4}{4}{12} & 67,890\\
\hline
\end{tabular}
\end{center}

Figure \ref{fig:results} graphically shows example solutions to the three benchmark problems. Each row depicts a single threshold function's weight vector, showing how much that threshold function depends on each of the original $N+M$ spins. 

\begin{figure}
   
    \includegraphics[width=0.99\linewidth]{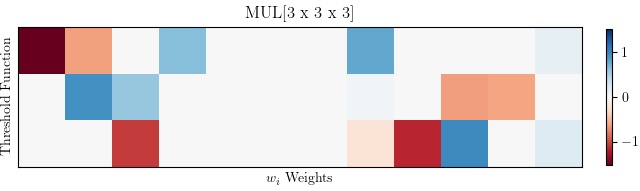}
    \includegraphics[width=0.99\linewidth]{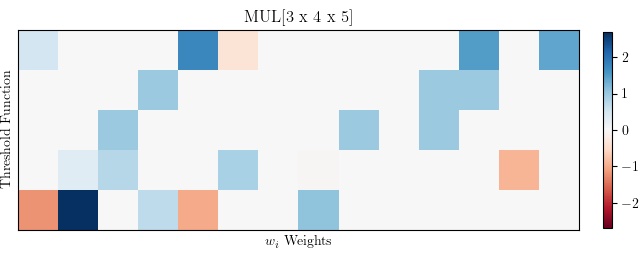}
    \includegraphics[width=0.99\linewidth]{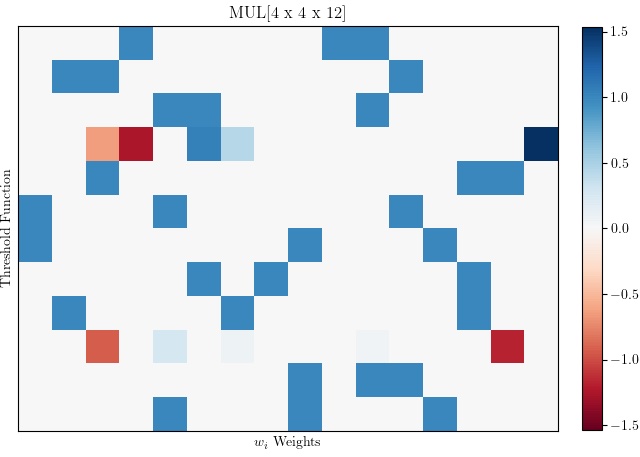}
    \caption{Examples of feasible sets of sparse auxiliary threshold functions for each of our benchmark problems. Each row is the weight vector $w$ of a threshold function.}
     \label{fig:results}
\end{figure}

\section{Conclusion}

Our main theoretical result has made computationally tractable the large-scale searches required to solve nontrivial cases of the Reverse Ising problem. As such, we can algorithmically discover far more compact Ising Hamiltonians which realize the desired circuity than would be possible to construct by hand. This has allowed us to significantly shrink the best known minimal-spin solutions to the Ising formulations of integer multiplication circuits, improving on previously known results by more than a factor of 3. However, we believe that doing better is possible. Relaxing to more general classes of auxiliary functions beyond tuples of threshold functions will likely allow for even smaller circuits, though it is much more mathematically difficult. Additionally, we wish to find ways to incorporate practical considerations such as energy gap, dynamic range, and graph structure into our model.

The practical feasibility of Ising computing, like other probabilistic, distributed, and quantum-inspired approaches, depends on the development of a strong theoretical foundation for the analysis of circuit design problems. We believe that this field is still in its infancy, that the mathematics is far from settled, and especially that we have only captured a fraction of the  the potential power of auxiliary maps for compactly realizing general functions. It is our hope that this paper will help to contribute both to the development of this theory and to the search for practical methods for programming the computers of the future.


\ifCLASSOPTIONcompsoc
  \section*{Acknowledgments}
\else
  \section*{Acknowledgment}
\fi

The authors would like to thank Dr. Karin Rabe and Dr. Gregory Moore (Department of Physics, Rutgers University) and Luisa Velasco (Department of Mathematics, University of Texas at Austin) for their feedback and editorial help.

\section{Appendix}

\subsection{Additional Mathematical Details}\label{subsec:moremath}
In this section we fill in relevant mathematical details which distract from the primary results but which are nonetheless relevant or necessary. The most vital of these is the proof of Theorem \ref{thm:augmented_constraints}.
\begin{proof}[Proof of Theorem \ref{thm:augmented_constraints}]
  Throughout this proof let $\sigma, \omega$ and $\eta$ denote elements in $\Sigma^N$, $\Sigma^M$ and $\Sigma^A$ respectively. Suppose first that the circuit $(N, M, f)$ is solvable by an Ising system $(X, H)$. Define $g:\Sigma^N\times \Sigma^M \to \Sigma^A$ to be the auxiliary component of the minimizer with respect to $N\cup M$:
  \begin{align}
    g(\sigma, \omega) := \argmin_{\eta\in \Sigma^A} H(\sigma, \omega, \eta).
  \end{align}
  By definition of $g$, the circuit $(N\cup M, A, g)$ is solvable by the Ising system $(X, H)$. This means for some $\eta' \in \Sigma^A$ we have that
  \begin{align}
    H(\sigma, \omega, \eta) > H(\sigma, f(\sigma), \eta') \geq H(\sigma, f(\sigma), g(\sigma, f(\sigma)))
  \end{align}
  for all $\eta \neq \eta'$. Hence $H$ satisfies the weak neutralizability condition.

  Since $(X,H)$ solves the circuit $(N, M, f)$,
  \begin{align}
    H(\sigma, \omega, \eta) > H(\sigma, f(\sigma), g(\sigma, f(\sigma)))
  \end{align}
  for all $\omega \neq f(\sigma)$ and all $\eta$, so in particular,
  \begin{align}
    H(\sigma, \omega, g(\sigma, \omega)) > H(\sigma, f(\sigma), g(\sigma, f(\sigma))).
  \end{align}
  Thus $(X, H)$ satisfies the $g$-augmented constraints.

  \vspace{1.5em}

  Now suppose that $g$ is an arbitrary function such that $(N\cup M, A, g)$ is an abstract circuit solvable by an Ising system $(X, R)$ whose Hamiltonian $R$ satisfies (\ref{eqn:weak-neutralizability}) and that $(X,S)$ is an Ising system with Hamiltonian $S$ which satisfies the $g$-augmented constraints. Consider the family of Ising Hamiltonians $H_\lambda = S + \lambda R$ parameterized by $\lambda\in \mathbb R$. We show that for sufficiently large $\lambda$, $H_\lambda$, $(X,H_\lambda)$ together with auxiliary array $g(\sigma) = g(\sigma, f(\sigma))$ satisfies the weak constraints and hence solves the circuit $(N,M,f)$.

  Fix $\sigma$ and $\omega \neq f(\sigma)$, and consider first the case that $\eta = g(\sigma, \omega)$. Then
  \begin{align}
  \begin{split} 
    &\phantom{\iff}H_\lambda(\sigma, \omega, \eta) - H_\lambda(\sigma, f(\sigma), g(\sigma)) > 0 \\
    &\iff S(\sigma, \omega, \eta) - S(\sigma, f(\sigma), g(\sigma)) 
    \\&\hspace{2em} \lambda (R(\sigma, \omega, \eta) - R(\sigma, f(\sigma), g(\sigma))) > 0 \\
    &\iff S(\sigma, \omega, g(\sigma, \omega)) - S(\sigma, f(\sigma), g(\sigma, f(\sigma))) \\
    &\hspace{1em}+ \lambda R(\sigma, \omega, g(\sigma, \omega)) - \lambda R(\sigma, f(\sigma), g(\sigma, f(\sigma))) > 0 \\ 
    &\iff S(\sigma, \omega, g(\sigma, \omega)) - S(\sigma, f(\sigma), g(\sigma, f(\sigma))) > 0.
  \end{split}
  \end{align}
  Note that the independence of the final biconditional above follows from the weak neutralizability of $R$. Now suppose that $\eta \neq g(\sigma, \omega)$. Set
  \begin{align}
  \alpha = \min_{\substack{\omega\in \Sigma^M \\ \omega \neq f(\sigma)}} R(\sigma, \omega, \eta) - R(\sigma, f(\sigma), g(\sigma, f(\sigma))),
  \end{align}
  noting that by (\ref{eqn:weak-neutralizability}), the assumption that $(X,R)$ solves $(N\cup M, A, g)$ and because $\eta \neq g(\sigma, \omega)$ we have
  \begin{align}
    R(\sigma, \omega, \eta) > R(\sigma, \omega, g(\sigma, \omega)) \geq R(\sigma, f(\sigma), g(\sigma))
  \end{align}
  which in turn implies that $\alpha > 0$. Additionally set
  \begin{align}
    \beta = \max_{\sigma\in \Sigma^X} S(\sigma, f(\sigma), g(\sigma, f(\sigma))) - S(\sigma, \omega, \eta).
  \end{align}
  We then have
  \[
   H_\lambda(\sigma, \omega, \eta) - H_\lambda(\sigma, f(\sigma), g(\sigma, f(\sigma))) > 0
    \]
    \[\iff\]
  \[
    \lambda ~>~ \frac{S(\sigma, f(\sigma), g(\sigma, f(\sigma))) - S(\sigma, \omega, \eta)}{R(\sigma, \omega, \eta) - R(\sigma, f(\sigma), g(\sigma, f(\sigma)))}.
  \]
  Choosing $\lambda > \beta/\alpha$ ensures this is satisfied for all $\sigma \in \Sigma^X$.
\end{proof}
\subsubsection{Ising Symmetry Proofs}\label{subsubsec:ising_symmetry_proofs}
Here we prove that spin actions and coordinate permutations are Ising circuit symmetries. We denote by $f_\alpha$ and $H_\alpha$ $\alpha f$ and $\alpha H$ respectively.
\begin{lem}\label{lem:spin_actions_are_nice}
    If $(N,M,f)$ is a circuit, $(X,H)$ an Ising system, and $\alpha\in \Sigma^X$ a spin action, then for $\sigma \in \Sigma^N$ and $\omega \in \Sigma^M$ we have $(\alpha H)(\sigma, \omega) = H(\alpha|_N\cdot \sigma, \alpha|_M\cdot \omega)$.
\end{lem}
\begin{proof}
    If we let $p$ be the coefficient vector of $H$, then
\begin{align}        
\begin{split}
    \langle \alpha u, v(\sigma, \omega))\rangle &=\langle v(\alpha) \cdot u, v(\sigma, \omega))\rangle \\
    &=\langle u, v(\alpha) \cdot v(\sigma, \omega))\rangle  \\
    &=\langle u, v(\alpha|_N\cdot \sigma, \alpha|_M\cdot \omega))\rangle.
\end{split}
\end{align}
\end{proof}
\begin{prop}\label{prop:spin_actions_are_symmetries}
    Spin actions are Ising circuit symmetries.
\end{prop}
\begin{proof}
    Fix a spin action $\alpha \in \Sigma^X$, a circuit $(N,M,f)$ and an Ising system $(X,H)$. We must show that whenever $(X, H)$ solves $(N,M,f)$, $(X,H_\alpha)$ solves $(N,M,f_\alpha)$. By Lemma \ref{lem:spin_actions_are_nice},
    \begin{align}
    \begin{split}
        &H_\alpha(\sigma, \omega) - H_\alpha(\sigma, f_\alpha(\sigma))\\
        &= \langle \alpha u, v(\sigma, \omega) - v(\sigma, (\alpha f)(\sigma))\rangle \\
        &=\langle v(\alpha) \cdot u, v(\sigma, \omega) - v(\sigma, \alpha|_M \cdot f(\alpha_N \cdot \sigma))\rangle \\
        &=\langle u, v(\alpha) \cdot v(\sigma, \omega) - v(\alpha)\cdot v(\sigma, \alpha|_M \cdot f(\alpha_N \cdot \sigma))\rangle \\
        &=\langle u, v(\alpha|_N\cdot \sigma, \alpha|_M\cdot \omega) - v(\alpha|_N\cdot \sigma, f(\alpha_N \cdot \sigma))\rangle \\
        &=H(\alpha|_N\cdot \sigma, \alpha|_M\cdot \omega) - H(\alpha|_N, f(\alpha_N\cdot \sigma)).
    \end{split}
    \end{align}
    Hence
    \begin{align}
        H_\alpha(\sigma, \omega) - H_\alpha(\sigma, f_\alpha(\sigma)) > 0
    \end{align}
    if and only if 
    \begin{align}
        H(\alpha|_N\cdot \sigma, \alpha|_M\cdot \omega) - H(\alpha|_N, f(\alpha_N\cdot \sigma)) > 0.
    \end{align}
    By noting that $\omega \neq f_\alpha(\sigma)$ if and only if $\alpha|_M\cdot \omega \neq f(\alpha|_N \cdot \sigma)$, we are done.
\end{proof}
\begin{prop}\label{prop:coordinate_permutations_are_symmetries}
    Coordinate permutations are Ising circuit symmetries.
\end{prop}
\begin{proof}
    Fix a circuit $(N,M,f)$, an Ising circuit $(X,H)$, a permutation $\alpha$ of $N$, and a permutation $\beta$ of $M$. Additionally let $B$, $B_\alpha$, and $B_\beta$ be the constraint matrices given by the circuits $(N,M,f)$, $(N,M,\alpha f)$, and $(N,M,\beta f)$ respectively (see Equations \ref{eqn:weak-constraints} and \ref{eqn:constraint_matrix} for the definition of these matrices).
    
    In Section \ref{subsubsec:coordinate_permutations} we defined the action of $\alpha$ on $f$ by $\alpha f := f\circ \alpha$. This results in a row permutation of $B$; that is, $B_\alpha$ is obtained by permuting the rows of $B$. This means any Ising system $(X,H)$ which solves $(N,M,f)$ also solves $(N,M,\alpha f)$ without modification.

    Recall we defined the action of $\beta$ on $f$ by $\beta f:= \beta \circ f$. The matrix $B_\beta$ is obtained by permuting the columns of $B$, so if $(X,H)$ solves $(N,M,f)$ we can obtain an Ising system $(X,\beta H)$ which solves $(N,M,\beta f)$ by applying the same permutation to the coefficient vector $p$ of $H$. Writing the coefficient vector $u$ using the $h$ and $J$ notation of Equation \ref{eqn:hJ_convention}, we see that
    \begin{align}
        \beta h_i := h_{\beta(i)}, \hspace{1em} \beta J_{ij} := J_{\beta(i), \beta(j)}
    \end{align}
    is the correct permutation. Note that we extend $\beta$ to a permutation on all of $X$ here by defining $\beta(i) := i$ for all $i\in X\setminus M$.
\end{proof}

\begin{rmk}\label{rmk:hyperoctahedral_group}
    It is worth noting that the group generated by spin actions and coordinate permutations on $X$ is precisely the \textit{hyperoctahedral group} $B_X$, the symmetry group of the hypercube of dimension $|X|$ \cite{kerber}. Thought of this way, $\mathcal G(N,M) = B_N\times B_M\times B_A$.
\end{rmk}

\subsection{Linear Programming Solver Details}\label{sec:solver}
\subsubsection{Problem Formulation} Our goal is to solve the linear programming problem 
\begin{align}
		\min_{\phi, \rho}\ \langle 1, \rho\rangle &\text{ s.t. } B\phi + \rho \geq v, \rho \geq 0
\end{align}
Note that this is equivalent to
\begin{align}
		\max_{\lambda, s}\ \langle b, \lambda \rangle &\text{ s.t. } A^T \lambda + s = c, s \geq 0
\end{align}
If we make the identification
\begin{align}
		&b = \begin{bmatrix}
				0\\
				-1\end{bmatrix} 
		&c = \begin{bmatrix}
				-v\\
				0\end{bmatrix}
		&&A^T = \begin{bmatrix}
				-B & -I\\
				0  & -I
		\end{bmatrix}
\end{align}
Since if $\lambda = (\lambda_1, \lambda_2)$ and $s = (s_1, s_2)$, then $A^T\lambda + s = c$ can be written as the pair of Equations $-B\lambda_1 - \lambda_2 + s_1 = -v$ and $-\lambda_2 + s_2 = 0$, which can be re-arranged as $s_2 = \lambda_2$ and $s_1 = B\lambda_1 + \lambda_2 - v$, so $s \geq 0$ actually means $\lambda_2 \geq 0$ and $B\lambda_1 + \lambda_2 \geq v$. This recovers our original problem with the identification $\phi := \lambda_1$, $\rho := \lambda_2$. A word on the dimensions. Suppose that $B \in \R^{m \times n}$, with $m \gg n$. Then $A^T \in \R^{2m \times n+m}$, so $\lambda, b \in \R^{n+m}$ and $s, c, x \in \R^{2m}$. We can also go ahead and set $r_b \leftarrow Ax - b$ and $r_c \leftarrow A^T\lambda + s - c$. 

\subsubsection{Solving Systems Involving a Matrix and its Transpose}

We now discuss solutions to systems of the form $(AKA^T) p = q$. Let $A$ be as defined above and $K$ be some general diagonal matrix with diagonal vector $k$ split into $k_1, k_2 \in \mathbb R^m$. The only actual systems of Equations that we need to solve in this algorithm will be of the form $(AKA^T)p = q$. We will derive a general algorithm for solving this system in an efficient manner by taking advantage of the structure of $A$. Note that $p,q \in \mathbb R^{n+m}$.
\begin{align}
		AKA^T = 
		\begin{bmatrix}
				-B^T & 0\\
				-I & -I
		\end{bmatrix}
		\begin{bmatrix}
				K_1 & 0\\
				0 & K_2
		\end{bmatrix}
		\begin{bmatrix}
				-B & -I\\
				0 & -I
		\end{bmatrix}\\
		= 
		\begin{bmatrix}
				B^TK_1B & B^TK_1\\
				K_1B & K_1 + K_2
		\end{bmatrix}
\end{align}
Now, we can apply block UDL-decomposition. It will be convenient if the matrix that we have to invert is actually the bottom right (since it is diagonal), and therefore we need the formula (with $Q := W - XZ^{-1}Y$):
\begin{align}
		\begin{bmatrix}
				W & X\\
				Y & Z
		\end{bmatrix}
		=
		\begin{bmatrix}
				I & XZ^{-1}\\
				0 & I
		\end{bmatrix}
		\begin{bmatrix}
				Q & 0\\
				0 & Z
		\end{bmatrix}
		\begin{bmatrix}
				I & 0\\
				Z^{-1}Y & I
		\end{bmatrix}
\end{align}
Applying this, we obtain (with $D := K_1$ and $Z := K_1 + K_2$ and $\Omega := B^T(D-D^2Z^{-1})B$)
\begin{align}
		\begin{bmatrix}
				I & B^T DZ^{-1}\\
				0 & I
		\end{bmatrix}
		\begin{bmatrix}
				    \Omega & 0\\
				0 & Z
		\end{bmatrix}
		\begin{bmatrix}
				I & 0\\
				DZ^{-1} B & I
		\end{bmatrix}\\
		= 
		\begin{bmatrix}
				I & B^T DZ^{-1}\\
				0 & I
		\end{bmatrix}
		\begin{bmatrix}
			 \Omega & 0\\
			DB & Z
		\end{bmatrix}
\end{align}
Now, consider the Equation
\begin{align}
		\begin{bmatrix}
				I & B^T DZ^{-1}\\
				0 & I
		\end{bmatrix}
		\begin{bmatrix}
			y_1\\
			y_2
		\end{bmatrix}
		=
		\begin{bmatrix}
			q_1\\
			q_2
		\end{bmatrix}
\end{align}
We get $y_2 = q_2$ and $q_1 = y_1 + B^TDZ^{-1} y_2 = y_1 + B^TDZ^{-1} q_2$, so $y_1 = q_1 - B^TDZ^{-1}q_2$. Now, we want to solve
\begin{align}
		\begin{bmatrix}
			\Omega & 0\\
			DB & Z
		\end{bmatrix}
		\begin{bmatrix}
			p_1\\
			p_2
		\end{bmatrix}
		=
		\begin{bmatrix}
			q_1 - B^TDZ^{-1}q_2\\
			q_2
		\end{bmatrix}
\end{align}
Clearly $DB p_1 + Zp_2 = q_2$ implies $p_2 = Z^{-1}(q_2 - DBp_1)$, so we have reduced the problem to solving the linear system $\Omega p_1 = q_1 - B^TDZ^{-1}q_2$. Since $n$ is in general fairly small, this is actually an easy system to solve. This leads us to the following algorithm:
\begin{align}
	p_1 &\leftarrow \Omega^{-1}(q_1 - B^TDZ^{-1}q_2)\\
	p_2 &\leftarrow Z^{-1}(q_2 - DBp_1)
\end{align}

\subsubsection{Solving the Main System}

We need to solve systems of the form
\begin{align}
		\begin{bmatrix}
				0 & A^T & I\\
				A & 0   & 0\\
				S & 0   & X
		\end{bmatrix}
		\begin{bmatrix}
				\Delta x\\
				\Delta \lambda\\
				\Delta s
		\end{bmatrix} 
		= 
		\begin{bmatrix}
				-r_c\\
				-r_b\\
				L
		\end{bmatrix}
\end{align}
Which, written out as Equations, is
\begin{align}
		A^T\Delta \lambda + \Delta s = -r_c\\
		S\Delta x + X\Delta s = L\\
		A\Delta x = -r_b
\end{align}
We can re-arrange to determine that $\Delta s = -r_c - A^T\Delta \lambda$ and $\Delta x = S^{-1}(L - X\Delta s) = S^{-1}(L - X(-r_c - A^T\Delta \lambda))$, upon which the last Equation becomes
\begin{align}
		AS^{-1}(L - X(-r_c - A^T\Delta \lambda)) = -r_b\\
		AS^{-1}L + AS^{-1}Xr_c + AS^{-1}XA^T\Delta \lambda = -r_b\\
		(AS^{-1}XA^T) \Delta \lambda = -r_b - AS^{-1}(L + Xr_c)\\
		(AS^{-1}XA^T) \Delta \lambda = -Ax + b - AS^{-1}(L + Xr_c)\\
		= b - A(x + S^{-1}(L + Xr_c))
\end{align}
This gives us the following algorithm:
\begin{align}
	\Delta \lambda &\leftarrow (AS^{-1}XA^T)^{-1}(b-A(x + S^{-1}(Xr_c + L)))\\
	\Delta s &\leftarrow -r_c - A^T\Delta \lambda\\
	\Delta x &\leftarrow S^{-1}(L - X\Delta s)
\end{align}

\subsubsection{Optimizations}

The matrix $M$ has around $50\%$ sparsity in practice and is a sign matrix (entries are $-1, 0$ or $1$). Therefore, it is natural to store it in CSC (Compressed Sparse Column) format with low integer precision. Profiling an implementation of the algorithm in C shows that the most expensive step by far is the calculation of the coefficient matrix for the system of Equations. Since it always has the structure $M^TDM$ for some diagonal matrix $D$, this comes down to calculating $n^2$ weighted column-column inner products of $M$, so CSC has a natural advantage over CSR. We can first observe that since this matrix is symmetric, only the upper triangle needs to be computed, which cuts the cost in half. Now, if $C_i$ is the $i$th column of $M$, we need a quick way to calculate $\langle C_i, dC_j\rangle$ for a coefficient $d$ which is not known ahead of time. Since the two columns are both sparse, and known ahead of time, we can precompute for each pair of columns $C_i, C_j$ a list $R_{ij}$ of the indices at which they are simultaneously nonzero, as well as a vector $S_{ij}$ of their products such that $S_{ij}(k) = (C_i)_{R_{ij}(k)}(C_j)_{R_{ij}(k)}$. That way, we can compute only exactly what really needs to be done, that is $\sum_k S_{ij}(k)d_{R_{ij}(k)}$. Since all of these operations are highly parallel, they lend themselves nicely to the use of a threadpool.



%

\IEEEtriggeratref{12}
\bibliography{citations}

\end{document}